\documentclass{amsart}

\usepackage{amsmath}

\usepackage{amsthm}
\usepackage{amsfonts}
\usepackage{dsfont}
\usepackage{enumitem}
\usepackage{color}

\newcommand{\reals}{\mathbb{R}}

\newcommand{\complex}{\mathbb{C}}

\newcommand{\bracketb}[1]{\Big[#1\Big]}

\newcommand{\paraa}[1]{\big(#1\big)}
\newcommand{\parab}[1]{\Big(#1\Big)}

\newcommand{\id}{\operatorname{id}}

\newcommand{\spacearound}[1]{\quad#1\quad}
\newcommand{\equivalent}{\spacearound{\Leftrightarrow}}
\renewcommand{\implies}{\spacearound{\Rightarrow}}
\newcommand{\qtext}[1]{\quad\text{#1}\quad}

\newcommand{\qand}{\qtext{and}}

\newtheorem{theorem}{Theorem}[section]
\newtheorem{corollary}[theorem]{Corollary}
\newtheorem{lemma}[theorem]{Lemma}
\newtheorem{proposition}[theorem]{Proposition}

\theoremstyle{definition}
\newtheorem{definition}[theorem]{Definition}
\theoremstyle{remark}
\newtheorem{remark}[theorem]{Remark}
\numberwithin{equation}{section}

\newcommand{\g}{\mathfrak{g}}
\newcommand{\A}{\mathcal{A}}
\newcommand{\eh}{\hat{e}}
\renewcommand{\d}{\partial}
\newcommand{\TSigma}{T\Sigma}
\newcommand{\hz}{h^{0}}
\newcommand{\hh}{\hat{h}}
\newcommand{\hsigma}{h_{\Sigma}}

\renewcommand{\mid}{\mathds{1}}
\renewcommand{\Im}{\operatorname{Im}}
\newcommand{\Gammat}{\tilde{\Gamma}}
\newcommand{\Wh}{\mathcal{W}_\hbar}

\newcommand{\Fh}{\mathfrak{F}_\hbar}
\newcommand{\E}{\mathcal{E}}
\newcommand{\F}{\mathcal{F}}
\newcommand{\Phib}{\bar{\Phi}}
\newcommand{\db}{\bar{\d}}
\newcommand{\htilde}{\tilde{h}}
\newcommand{\Der}{\operatorname{Der}}
\newcommand{\Stwoh}{S^2_{\hbar}}
\newcommand{\TStwoh}{TS^2_{\hbar}}
\newcommand{\eps}{\varepsilon}
\newcommand{\et}{\tilde{e}}
\newcommand{\Nh}{\mathcal{N}_\hbar}
\newcommand{\nablat}{\tilde{\nabla}}
\newcommand{\ph}{\hat{p}}
\newcommand{\gammat}{\tilde{\gamma}}
\newcommand{\xv}{\vec{x}}
\newcommand{\tr}{\operatorname{tr}}

\title[Levi-Civita connections for noncommutative minimal surfaces]{Levi-Civita connections for a class of noncommutative minimal
  surfaces}

\author{Joakim Arnlind}

\address[Joakim Arnlind]{Dept. of Math.\\
Link\"oping University\\
581 83 Link\"oping\\
Sweden}
\email{joakim.arnlind@liu.se}

\subjclass[2000]{}
\keywords{}

\begin{document}

\maketitle

\begin{abstract}
  We study connections on hermitian modules, and show that metric
  connections exist on regular hermitian modules; i.e finitely
  generated projective modules together with a non-singular hermitian
  form. In addition, we develop an index calculus for such modules, and
  provide a characterization in terms of the existence of a
  pseudo-inverse of the matrix representing the hermitian form with
  respect to a set of generators. As a first illustration of the above
  concepts, we find metric connections on the fuzzy sphere. Finally,
  the framework is applied to a class of noncommutative minimal
  surfaces, for which there is a natural concept of torsion, and we
  prove that there exist metric and torsion free connections for
  every minimal surface in this class.
\end{abstract}

\section{Introduction}

\noindent
In the last decade, a lot of progress has been made in understand the
Riemannian aspects of noncommutative geometry. Both from the viewpoint
of spectral triples (see
e.g. \cite{ct:gaussBonnet,fk:scalarCurvature}) and quantum groups (see
e.g. \cite{bm:starCompatibleConnections,mw:quantum.koszul,bgl:lc.nc.differential.calculi})
as well as from a derivation based approach (see
e.g. \cite{ac:ncgravitysolutions,r:leviCivita,aw:cgb.sphere,aw:curvature.three.sphere,a:lc.braided}).

A classical subject in Riemannian geometry is the theory of minimal
submanifolds; i.e. embeddings into a Riemannian manifold such that the
area functional of the embedded manifold has a stationary point (with
respect to the induced metric).  In particular, the theory of minimal
surfaces is a fascinating subject that is, after more than a hundred
years, still an active field of research. It is interesting to ask if
there is a non-trivial theory of minimal submanifolds in
noncommutative geometry. This question has been studied from several
different perspectives (see e.g.
\cite{mr:nc.sigma.model,dll:sigma.model.solitons,l:twisted.sigma.model}).
In the context of matrix models, related to string and membrane
theories, one is interested in finding noncommutative minimal surfaces
as solutions to certain equations that are noncommutative analogues of
equations for harmonic maps (see
e.g. \cite{h:phdthesis,ct:holomorphic.curves.matrices,aht:spinning}). However,
finding explicit examples seems to be a notoriously hard problem,
although progress has been made (see
e.g. \cite{ah:dmsa,ah:quantizedMinimal,ach:nms,s:fuzzy.construction.kit,ss:covariant.spheres,ah:catenoid,ahk:qms,atn:minimal.embeddings}).

In \cite{ach:nms} a noncommutative analogue of minimal surfaces in
$\reals^n$ was developed and it was shown that one can explicitly
generate infinite classes of noncommutative minimal
surfaces. Moreover, one can a Weierstrass theorem completely
characterizing all noncommutative minimal surfaces in $\reals^3$.  However,
the Riemannian aspects of these surfaces were not studied in detail.
In this note, we would like to study the existence of Levi-Civita
connections (i.e. metric and torsion free connections) on a class of
noncommutative minimal surfaces, as defined in \cite{ach:nms}, built
from the Weyl algebra together with a choice of algebra elements
representing the embedding coordinates. To achieve this, we start by
studying connections on projective modules that are equipped with a
non-singular hermitian form, so called \emph{regular hermitian
  modules} (cf. Definition~\ref{def:regular.hermitian.module}). We
show that on every such module, there exist connections that are
compatible with the hermitian form
(cf. Theorem~\ref{thm:M.projective.embedded}), and explicit
expressions for such connections are derived. Moreover. we
characterize regular hermitian modules in terms of the existence of a
(pseudo) inverse to the matrix representing a hermitian form with
respect to a set of generators
(cf. Proposition~\ref{prop:huab.equiv.projective}), providing a
convenient way of doing ``index'' computations in hermitian modules.
As an illustration of these concepts, we construct metric connections
on the fuzzy sphere in Section~\ref{sec:fuzzy.sphere}.

Finally, we recall the concept of a noncommutative minimal surface
from \cite{ach:nms} and show that they generate regular hermitian
modules for which there is a natural concept of torsion, and prove
that every regular noncommutative minimal surface has a metric and
torsion free connection.

\section{Hermitian modules}

\noindent
In this section, we will be concerned with hermitian modules; i.e
modules that are equipped with an algebra valued bilinear (hermitian)
form.  In particular, we will study \emph{regular hermitian modules},
defined as (finitely generated) projective modules together with an
invertible hermitian form. In a geometric setting, they correspond to
vector bundles with a metric structure.

In what follows, we assume $\A$ to be a unital $\ast$-algebra over
$\complex$. A left (right) $\A$-module $M$ has a canonical structure
of a right (left) $\A$-module, by defining $m\cdot a=a^\ast\cdot m$
(resp. $a\cdot m=m\cdot a^\ast$). Note that, in general, this does not
provide a bimodule structure on $M$.  For a right (left) $\A$-module
$M$ we denote the dual module by $M^\ast$. The dual module is
canonically a left (right) module by setting
\begin{align*}
  (a\cdot\omega)(m) = a\omega(m)\qquad\paraa{\text{resp. }
  (\omega\cdot a)(m) = \omega(m)a}
\end{align*}
for $\omega\in M^\ast$, $m\in M$ and $a\in\A$. In the following, we
shall also consider $M^\ast$ as a right (left) module by the previous
construction; i.e. $(\omega\cdot a)(m)=a^\ast\omega(m)$
(resp. $(a\cdot\omega)(m)=\omega(m)a^\ast$). Let us now recall the
definition of a hermitian form.

\begin{definition}
  Let $\A$ be a $\ast$-algebra and let $M$ be a right (left)
  $A$-module. A \emph{hermitian form on $M$} is a map
  $h:M\times M\to\A$ satisfying
  \begin{enumerate}
  \item $h(m_1,m_2)^\ast = h(m_2,m_1)$
  \item $h(m_1,m_2a)=h(m_1,m_2)a$ \quad(resp. $h(am_1,m_2)=ah(m_1,m_2)$)
  \item $h(m_1+m_2,m)=h(m_1,m)+h(m_2,m)$
  \end{enumerate}
  for all $m,m_1,m_2\in M$ and $a\in\A$.
\end{definition}

\noindent
For instance, on a free right $\A$-module $\A^n$ with basis $\{\eh_i\}_{i=1}^n$, a hermitian
form is determined by $h_{ij}\in\A$ (for $i,j=1,\ldots,n$) as
\begin{align*}
  h(\eh_iU^i,\eh_iV^i) = (U^i)^\ast h_{ij}V^j.
\end{align*}
Here, and in the following, we adopt the convention that repeated
indices are summed over an appropriate range; e.g.
\begin{align*}
  \eh_iU^i = \sum_{i=1}^n\eh_iU^i\qand
  (U^i)^\ast h_{ij}V^j = \sum_{i,j=1}^n(U^i)^\ast h_{ij}V^j.
\end{align*}
To a hermitian form $h$ on a right (left) $\A$-module $M$, one
associates the map $\hh:M\to M^\ast$ given by
\begin{align}
  \hh(m_1)(m_2) = h(m_1,m_2)\qquad
  \paraa{\text{resp. } \hh(m_1)(m_2)=h(m_2,m_1)},
\end{align}
and it follows immediately that $\hh$ satisfies
\begin{align}
  &\hh(m_1+m_2) = \hh(m_1)+\hh(m_2)\\
  &\hh(m_1a) = a^\ast\hh(m_1)\qquad
    \paraa{\text{resp. }\hh(am_1)=\hh(m_1)a^\ast}.
\end{align}
Thus, considering $M^\ast$ as a right (left) module, $\hh$ is a
homomorphism of right (left) modules.

\begin{definition}
  A hermitian form $h$ on the $\A$-module $M$ is called
  \emph{invertible} if $\hh:M\to M^\ast$ is a bijection, and one
  defines $h^{-1}:M^\ast\times M^\ast\to\A$ as
  \begin{align*}
   h^{-1}(\omega_1,\omega_2)=\omega_1\paraa{\hh^{-1}(\omega_2)}. 
  \end{align*}
\end{definition}

\begin{proposition}\label{prop:hinv.hermitian.form}
  If $h$ is an invertible hermitian form on the right $\A$-module $M$
  then $h^{-1}$ is a hermitian form on the left $\A$-module $M^\ast$.
\end{proposition}

\begin{proof}
  First of all, it is clear that $h^{-1}$ is linear in its first
  argument. Secondly,
  \begin{align*}
    h^{-1}(a\omega_1,\omega_2) = (a\omega_1)\paraa{\hh^{-1}(\omega_2)}
    =a\omega_1\paraa{\hh^{-1}(\omega_2)} = ah^{-1}(\omega_1,\omega_2).
  \end{align*}
  Furthermore, writing $\omega_1=\hh(m_1)$ and $\omega_2=\hh(m_2)$ one finds that
  \begin{align*}
    h^{-1}(\omega_1,\omega_2)^{\ast} 
    &= h^{-1}\paraa{\hh(m_1),\hh(m_2)}^\ast =
      \paraa{\hh(m_1)(m_2)}^\ast = h(m_1,m_2)^\ast\\
    &= h(m_2,m_1) = \hh(m_2)(m_1) =
      h^{-1}\paraa{\hh(m_2),\hh(m_1)} = h^{-1}(\omega_2,\omega_1).
  \end{align*}
  We conclude that $h^{-1}$ is a hermitian form on the left
  $\A$-module $M$. 
\end{proof}

\noindent
It follows directly from Proposition~\ref{prop:hinv.hermitian.form}
that $h^{-1}$ is also a hermitian form on the \emph{right} $\A$-module
$M^\ast$; namely
\begin{align*}
  h^{-1}(\omega,\eta\cdot a) = h^{-1}(a^\ast\eta,\omega)^\ast
  = \paraa{a^\ast h^{-1}(\eta,\omega)}^\ast = h^{-1}(\omega,\eta)a.
\end{align*}
Next, let us recall the definition of a hermitian module. 

\begin{definition}
  Let $M$ be a right (left) $\A$-module and let $h$ be a hermitian
  form on $M$. The pair $(M,h)$ is called a \emph{right (left)
    hermitian module}. Moreover, a homomorphism of right (left)
  hermitian modules $\phi:(M,h)\to(M',h')$ is a homomorphism of right
  (left) modules $\phi:M\to M'$ such that
  $h(m_1,m_2)=h'(\phi(m_1),\phi(m_2))$ for all $m_1,m_2\in
  M$. Moreover, if $\phi$ is a module isomorphism then $\phi$ is
  called an \emph{isometry} and $(M,h)$ and $(M',h')$ are said to be
  \emph{isometric} and we write $(M,h)\simeq (M',h')$.
\end{definition}

\noindent
The next definition introduces a nice class of hermitian modules where
the module is assumed to be a (finitely generated) projective module,
and the hermitian form is assumed to be invertible. They correspond in
a stronger sense to classical vector bundles with metrics.

\begin{definition}\label{def:regular.hermitian.module}
  If $h$ is an invertible hermitian form on a finitely generated
  projective $\A$-module $M$, then $(M,h)$ is called a \emph{regular
    hermitian $\A$-module}.
\end{definition}

\noindent
In Riemannian geometry (and in differential geometry in general),
computations are often done in local coordinates, and an immediate
consequence of the non-degeneracy of a Riemannian metric $g$ is that
there exists an inverse $g^{ab}$ (at each point) such that
$g^{ab}g_{bc}=\delta^a_c$, where $g_{ab}=g(\d_a,\d_b)$ and
$\{\d_a\}_{a=1}^n$ is a local basis of the tangent space. In general,
a global inverse does not exist, unless the tangent bundle is
trivial. However, the next result shows that, in the algebraic setting, there
exists an inverse relative to a set of generators if and only if the
hermitian module is regular.

\begin{proposition}\label{prop:huab.equiv.projective}
  Let $(M,h)$ be a right hermitian $\A$-module with generators
  $\{e_a\}_{a=1}^n$, and set $h_{ab}=h(e_a,e_b)$. Then $(M,h)$ is a
  regular hermitian module if and only if there exist $h^{ab}\in\A$
  (for $a,b=1,\ldots,n$) such that $(h^{ab})^\ast=h^{ba}$ and $e_rh^{rs}h_{sc}=e_c$ for
  $a,b,c=1,\ldots,n$.
\end{proposition}

\begin{proof}
  Assume that there exists $h^{ab}$ such that $(h^{ab})^\ast=h^{ba}$
  and $e_ah^{ab}h_{bc}=e_c$. Let us start by proving that $M$ is
  projective by showing that for any surjective homomorphism
  $\phi:N\to M$ there exists a homomorphism $\psi:M\to N$ such that
  $\phi\circ\psi=\id_M$. Thus, if $\phi:N\to M$ is a surjective
  homomorphism then there exists $f_a\in N$ such that $\phi(f_a)=e_a$
  for $a=1,\ldots,n$. Define $\psi:M\to N$ as
  \begin{align*}
    \psi(m) = f_ah^{ab}h(e_b,m),
  \end{align*}
  which is clearly a homomorphism, and it is easy
  to check that (with $m=e_cm^c$)
  \begin{align*}
    \phi(\psi(m)) = \phi(f_a)h^{ab}h(e_b,e_c)m^c
    = e_ah^{ab}h_{bc}m^c=e_cm^c=m,
  \end{align*}
  by using that $e_ah^{ab}h_{bc}=e_c$. Hence, $M$ is a finitely
  generated projective module.
  
  Next, let us show that $h$ is invertible, and we start by showing
  that $\hh$ is injective.  To this end, assume that $\hh(m)=0$, which
  is equivalent to $h(n,m)=0$ for all $n\in M$. Choosing
  $n=e_a(h^{ca})^\ast$ one obtains
  \begin{align*}
    &h^{ca}h(e_a,e_b)m^b = 0\implies
    h^{ca}h_{ab}m^b = 0\implies
      e_ch^{ca}h_{ab}m^b = 0\implies\\
    &e_bm^b = 0\implies m=0
  \end{align*}
  by using that $e_ch^{ca}h_{ab}=e_b$. Hence, $\hh$ is injective. Let
  us now show that $\hh$ is surjective. Let $\omega\in M^\ast$ be
  arbitrary. Choosing $m=e_a(\omega(e_c)h^{ca})^\ast$ one obtains
  \begin{align*}
    \hh(m)(e_bn^b)
    &=\omega(e_c)h^{ca}h\paraa{e_a,e_b}n^b
      = \omega(e_c)h^{ca}h_{ab}n^b
      = \omega(e_ch^{ca}h_{ab})n^b\\
    &=\omega(e_b)n^b=\omega(e_bn^b)
  \end{align*}
  showing that $\hh(m)=\omega$. We conclude that $\hh$ is surjective,
  implying that $h$ is invertible. This completes the proof that if
  such an $h^{ab}$ exists then $M$ is a finitely generated projective
  module and $h$ is invertible.

  Conversely, assume that $M$ is a finitely generated projective
  module and assume that $h$ is invertible.  Since $M$ is a finitely
  generated projective module, there exists a dual basis
  $\{f_a\}_{a=1}^n\subseteq M^\ast$ such that $m = e_af^a(m)$ for all
  $m\in M$, and
  \begin{align*}
    (h_{ab}f^b)(m) = h(e_a,e_b)f^b(m) = h(e_a,e_bf^b(m))
    = h(e_a,m) = \hh(e_a)(m)
  \end{align*}
  shows that $\hh(e_a)=h_{ab}f^b$. Next, let $h^{ab}=h^{-1}(f^a,f^b)$,
  satisfying $(h^{ab})^\ast=h^{ba}$, and compute
  \begin{align*}
    e_ah^{ab}h_{bc}
    &= e_ah^{-1}(f^a,f^b)h_{bc}
      = e_af^a\paraa{\hh^{-1}(f^b)}h_{bc}
      = e_af^a\paraa{\hh^{-1}(f^b)h_{bc}}\\
    &= e_af^a\paraa{\hh^{-1}(h_{cb}f^b)}
      =e_af^a\paraa{\hh^{-1}(\hh(e_c))}
      =e_af^a(e_c) = e_c,
  \end{align*}
  which concludes the proof.  
\end{proof}

\noindent
It follows directly from Proposition~\ref{prop:huab.equiv.projective}
that a global inverse exists precisely when the module is free.

\begin{corollary}
  Let $(M,h)$ be a right hermitian $\A$-module with generators
  $\{e_a\}_{a=1}^n$, and set $h_{ab}=h(e_a,e_b)$. Then $(M,h)$ is a
  regular hermitian module and $\{e_a\}_{a=1}^n$ is a basis of $M$ if and only
  if there exist $h^{ab}\in\A$ (for $a,b=1,\ldots,n$) such that
  $(h^{ab})^\ast=h^{ba}$ and $h^{ap}h_{pc}=\delta^a_c\mid$ for
  $a,b,c=1,\ldots,n$.
\end{corollary}

\begin{proof}
  First, assume that $\{e_a\}_{a=1}^n$ is a basis of $M$ and that $h$
  is invertible. In particular, $M$ is a projective module, and
  Proposition~\ref{prop:huab.equiv.projective} implies that there
  exists $h^{ab}$ such that $(h^{ab})^\ast=h^{ba}$ and
  $e_ah^{ab}h_{bc}=e_c$. Since $\{e_a\}_{a=1}^n$ is a basis of $M$, it
  follows that $h^{ab}h_{bc}=\delta^a_c\mid$. Conversely, assume that
  $(h^{ab})^\ast=h^{ba}$ and $h^{ab}h_{bc}=\delta^a_c\mid$, which in
  particular implies that $e_ah^{ab}h_{bc}=e_c$. It follows from
  Proposition~\ref{prop:huab.equiv.projective} that $h$ is
  invertible. Moreover, one checks that
  \begin{align*}
    &e_am^a = 0\implies
    h(e_b,e_a)m^a=0\implies
      h_{ba}m^a=0\implies\\
    &h^{cb}h_{ba}m^a=0\implies
      m^c=0,
  \end{align*}
  showing that $\{e_a\}_{a=1}^n$ is indeed a basis of $M$.
\end{proof}

\noindent
One of the technical benefits of projective modules, is that they can
be realized as the image of a projection in a free module, allowing
one to construct objects in a projective module by first constructing
them in a free module (which is in general much simpler). A typical
example, which is relevant for this paper, are connections on a
projective module, which can be obtained by projecting connections on
the corresponding free module.

Considering a hermitian module $(M,h)$, such that $M$ is projective,
it is natural to ask if every such hermitian module is isometric to
the image of a projection in a free module; i.e. if there exists a
free hermitian module $(\A^n,\htilde)$ and a projection
$p:\A^n\to\A^n$ such that $(p(\A^n),\htilde|_{p(\A^n)})\simeq (M,h)$?
A positive answer to this question can be found in
Theorem~\ref{thm:M.projective.embedded}, but let us first start by
showing how one can easily construct such modules by using orthogonal
projections.

\begin{definition}
  Let $(M,h)$ be a hermitian $\A$-module and let $\phi:M\to M$ be an
  endomorphism. If $h(\phi(m_1),m_2)=h(m_1,\phi(m_2))$ for all
  $m_2,m_2\in M$ then $\phi$ is said to be \emph{orthogonal with respect
    to $h$}.
\end{definition}

\begin{proposition}\label{prop:projection.regular.hermitian.module}
  Let $(\A^n,\htilde)$ be a free regular hermitian module. If $p$ is an orthogonal
  projection on $(\A^n,\htilde)$, then
  $(p(\A^n),\htilde|_{p(\A^n)})$ is a regular hermitian module.
\end{proposition}

\begin{proof}
  Setting $M=p(\A^n)$, one identifies the elements of $M$ with
  elements $U\in\A^n$ such that $p(U)=U$. Analogously, the elements of
  the dual module $M^\ast$ can be identified with elements
  $\omega\in(\A^n)^\ast$ such that $p^t(\omega)=\omega$, where
  $p^t:(\A^n)^\ast\to(\A^n)^\ast$ denotes the transpose of $p$;
  i.e. $p^t(\omega)(U)=\omega(p(U))$. Setting $h=\htilde|_{p(\A^n)}$
  it is easy to check that $p^t(\hh(m))=\hh(m)$ for all $m\in M$; namely
  \begin{align*}
    p^t(\hh(m))(V) = \hh(m)\paraa{p(V)} = h(m,p(V))
    =h(p(m),V)=h(m,V) = \hh(m)(V),
  \end{align*}
  since $p$ is an orthogonal projection. Let us now show that $\hh$ is
  injective. Thus, assume $\hh(m)=0$ for some $m\in M$, which is
  equivalent to $\hh(m)(p(U))=0$ for all $U\in\A^n$. Since
  $p^t(\hh(m))=\hh(m)$ one obtains
  \begin{align*}
    0 = \hh(m)\paraa{p(U)} = p^t\paraa{\hh(m)}(U) = \hh(m)(U) = \htilde(m,U)
  \end{align*}
  for all $U\in\A^n$, implying that $m=0$ since $\htilde$ is
  invertible. Hence, $\hh$ is injective. Let us now prove that $\hh$
  is surjective. Thus, let $\omega\in M^\ast$ be arbitrary. Since
  $\htilde$ is invertible, there exists $U\in\A^n$ such that
  $\hat{\htilde}(U)=\omega$. Since $p^t(\omega)=\omega$ is follows
  that
  \begin{align*}
    p^t\paraa{\hat{\htilde}(U)} = \hat{\htilde}(U)\equivalent
    \htilde(p(U),V) = \htilde(U,V)
  \end{align*}
  for all $V\in\A^n$. Since $\htilde$ is invertible, we conclude that
  $p(U)=U$ which shows that there exists $U\in M$ such that
  $\hh(U)=\omega$. Hence, $h=\htilde|_{p(\A^n)}$ is invertible and
  $(p(\A^n),h)$ is a regular hermitian module.
\end{proof}

\noindent
The above result tells us that as soon as one has an orthogonal
projection and an invertible metric on a free module, one may
construct a regular hermitian module as the image of the
projection. The next result shows that this situation is in fact
generic; every regular hermitian module arise in this way.

\begin{theorem}\label{thm:M.projective.embedded}
  Let $(M,h')$ be a regular hermitian
  right $\A$-module such that $M$ is generated by $n$
  elements. Moreover, define the following invertible hermitian form\footnote{Note that in the direct sum $\A^n\oplus(\A^n)^\ast$, the dual $(\A^n)^\ast$ is considered as a right $\A$-module.}
  \begin{align*}
    &b:\A^n\oplus(\A^n)^\ast\times \A^n\oplus(\A^n)^\ast\to\A\\
    &b\paraa{(U,\omega),(V,\eta)} = \eta(U)^\ast + \omega(V)
  \end{align*}
  for $U,V\in\A^n$ and $\omega,\eta\in(\A^n)^\ast$. Then there exists
  a projection
  \begin{align*}
   \ph:\A^n\oplus(\A^n)^\ast\to \A^n\oplus(\A^n)^\ast, 
  \end{align*}
  which is orthogonal with respect to $b$, such that
  \begin{align*}
    \paraa{\ph(\A^n\oplus(\A^n)^\ast),b|_{\ph(\A^n\oplus(\A^n)^\ast)}}
    \simeq (M,h').
  \end{align*}
\end{theorem}

\begin{proof}
  Let $\{e_i\}_{i=1}^n$ be generators of $M$ and let
  $\{\eh_i\}_{i=1}^n$ be a basis of $\A^n$; moreover, one defines
  $\phi:\A^n\to M$ as
  \begin{align*}
    \phi(\eh_iU^i)=\phi(e_i)U^i,
  \end{align*}
  inducing the hermitian form $h:\A^n\times\A^n\to\A$ as
  \begin{align}\label{eq:induced.metric.def}
    h(U,V) = h'\paraa{\phi(U),\phi(V)}.
  \end{align}
  Since $M$ is a projective module (and $\phi$ is surjective), there
  exists a module homomorphism $\psi:M\to\A^n$ such that
  $\phi\circ\psi=\id_M$. Defining $p=\psi\circ\phi$, it is a standard
  fact that that $p^2=p$ and $p(\A^n)\simeq M$, and from
  \eqref{eq:induced.metric.def} it follows that $(p(\A^n),h|_{p(\A^n)})$
  is in fact isometric to $(M,h')$.  Furthermore, the projection is
  orthogonal with respect to $h$; i.e.
  \begin{align*}
    h\paraa{p(U),V}
    &= h\paraa{(\psi\circ\phi)(U),V}
      =h'\paraa{(\phi\circ\psi\circ\phi)(U),\phi(V)}
      =h'\paraa{\phi(U),\phi(V)}\\
    &=h'\paraa{\phi(U),(\phi\circ\psi\circ\phi)(V)}
      =h\paraa{U,(\psi\circ\phi)(V)} = h\paraa{U,p(V)},
  \end{align*}
  due to $\phi\circ\psi=\id_M$.
 
  In the following, we shall construct a projection $\ph$ on the free
  module
  \begin{align*}
    B=\A^n\oplus(\A^n)^\ast,
  \end{align*}
  that is orthogonal with respect to the invertible hermitian form $b$
  on $B$, such that
  \begin{align*}
    \paraa{\ph(B),b|_{\ph(B)}} \simeq \paraa{p(\A^n),h|_{p(\A^n)}}\simeq(M,h').
  \end{align*}
  Recall that the transpose of $p$, denoted by
  $p^t:(\A^n)^\ast\to(\A^n)^\ast$, is the map given by
  \begin{align*}
    (p^t\omega)(U) = \omega\paraa{p(U)}
  \end{align*}
  and it follows that $(p^t)^2=p^t$. Moreover, note that
  \begin{align}\label{eq:pthh}
    (p^t\circ\hh\circ p)(U) = \hh(p^2(U)) = \hh(p(U)) = (\hh\circ p)(U),
  \end{align}
  for all $U\in\A^n$, as well as
  \begin{align*}
    \paraa{(\hh\circ p\circ\hh^{-1}\circ p^t)(\omega)}(U)
    &=h\paraa{(p\circ\hh^{-1}\circ p^t)(\omega),U}
    =h\paraa{(\hh^{-1}\circ p^t)(\omega),p(U)}\\
    &=h\paraa{p(U),(\hh^{-1}\circ p^t)(\omega)}^\ast
      =\bracketb{\hh\paraa{p(U)}\paraa{(\hh^{-1}\circ p^t)(\omega)}}^\ast\\
    &=(p^t\omega)\paraa{p(U)}
      =\omega\paraa{p^2(U)}=\omega\paraa{p(U)}
      =(p^t\omega)(U)
  \end{align*}
  (using that $\eta(V)=\hh(V)(\hh^{-1}(\eta))^\ast$), from which we
  conclude that
  \begin{align}\label{eq:phhinv}
    \paraa{p\circ\hh^{-1}\circ p^t}(\omega) = \paraa{\hh^{-1}\circ p^t}(\omega).
  \end{align}
  for all $\omega\in(\A^n)^\ast$. Now, define
  $\ph:B\to B$ as
  \begin{align*}
    \ph(U,\omega) = \frac{1}{2}
    \paraa{p(U)+(\hh^{-1}\circ p^t)(\omega),p^t(\omega)+(\hh\circ p)(U)}
  \end{align*}
  for $(U,\omega)\in B$. One computes
  \begin{align*}
    \ph^2(U,\omega)
    &= \frac{1}{4}
    \parab{p\paraa{p(U)+(\hh^{-1}\circ p^t)(\omega)}+
    (\hh^{-1}\circ p^t)\paraa{p^t(\omega)+(\hh\circ p)(U)},\\
    &\qquad\qquad p^t\paraa{p^t(\omega)+(\hh\circ p)(U)}+
      (\hh\circ p)\paraa{p(U)+(\hh^{-1}\circ p^t)(\omega)}}\\
    &= \frac{1}{4}\parab{
      2p(U)+2(\hh^{-1}\circ p^t)(\omega),
      2p^t(\omega)+2(\hh\circ p)(U)
      }=\ph(U,\omega),
  \end{align*}
  by using that $p^2=p$, $(p^t)^2=p^t$ together with \eqref{eq:pthh}
  and \eqref{eq:phhinv}.
  
  Next, let $b:B\times B\to \A$ be defined by
  \begin{align*}
    b\paraa{(U,\omega),(V,\eta)} = \eta(U)^\ast + \omega(V),
  \end{align*}
  and it is easy to check that $b$ is an invertible hermitian form on
  $B$. Let us now show that $\ph$ is orthogonal with respect to $b$. One computes
  \begin{align*}
    b\paraa{&\ph(U,\omega),(V,\eta)}
              = b\paraa{(p(U)+(\hh^{-1}\circ p^t)(\omega),p^t(\omega)+(\hh\circ p)(U)),(V,\eta)}\\
            &= \eta\paraa{p(U)}^\ast+\eta\paraa{(\hh^{-1}\circ p^t)(\omega)}^\ast
              +(p^t\omega)(V)+\paraa{(\hh\circ p)(U)}(V)\\
            &= (p^t\eta)(U)^\ast+h^{-1}\paraa{\eta,p^t\omega}^\ast
              +\omega\paraa{p(V)}+h\paraa{p(U),V}\\
            &= (p^t\eta)(U)^\ast+h^{-1}\paraa{\omega,p^t\eta}
              +\omega\paraa{p(V)}+h(p(V),U)^\ast\\
            &= (p^t\eta)(U)^\ast + \omega\paraa{(\hh^{-1}\circ p^t)(\eta)}
              +\omega\paraa{p(V)}+h(p(V),U)^\ast\\
            &= b\paraa{(U,\omega),\ph(V,\eta)},
  \end{align*}
  by using that $h^{-1}(p^t\omega,\eta)=h^{-1}(\omega,p^t\eta)$, which
  follows from the fact that $p$ is orthogonal with respect to
  $h$. Finally, let us show that $(\ph(B),b|_{\ph(B)})$ is isometric
  to $(p(\A^n),h|_{p(\A^n)})$. Defining a module homomorphism $\Phi:p(\A^n)\to B$ as
  \begin{align*}
    \Phi(U) = \tfrac{1}{\sqrt{2}}(U,\hh(U))
  \end{align*}
  and one readily checks that $\ph(\Phi(U))=\Phi(U)$ for all
  $U\in\A^n$ such that $p(U)=U$, implying that
  $\Phi:p(\A^n)\to \ph(B)$. Moreover, $\Phi$ is injective since
  $\Phi(U)=(0,0)$ immediately implies that $U=0$. Note that, for
  $U=p(V)+(\hh^{-1}\circ p^t)(\eta)\in p(\A^n)$ one obtains
  \begin{align*}
    \Phi(U) = \tfrac{1}{\sqrt{2}}
    \paraa{p(V)+(\hh^{-1}\circ p^t)(\eta),(\hh\circ p)(V)+p^t(\eta)}
    =\sqrt{2}\,\ph(V,\eta)
  \end{align*}
  showing that $\Phi:\Phi:p(\A^n)\to \ph(B)$ is surjective. Hence,
  $\Phi$ is a module isomorphism. Moreover, one finds that
  \begin{align*}
    b\paraa{\Phi(U),\Phi(V)}
    &= \tfrac{1}{2}b\paraa{(U,\hh(U)),(V,\hh(V))}
      =\tfrac{1}{2}\paraa{\hh(V)(U)^\ast+\hh(U)(V)}\\
    &=\tfrac{1}{2}\paraa{h(V,U)^\ast+h(U,V)}=h(U,V)
  \end{align*}
  showing that $\Phi$ is indeed an isometry. We conclude that
  \begin{equation*}
    \paraa{\ph(B),b|_{\ph(B)}} \simeq \paraa{p(\A^n),h|_{p(\A^n)}}\simeq(M,h').\qedhere    
  \end{equation*}
\end{proof}

\noindent
In the next section, we shall use
Theorem~\ref{thm:M.projective.embedded} to show that every regular
hermitian module has a connection which is compatible with the
hermitian form.

\section{Lie pairs and affine connections}
\label{sec:lie.pairs}

\noindent
Given a hermitian module $(M,h)$ we are interested in studying
connections on $M$ that are compatible with the hermitian form $h$,
much in analogy with hermitian connections on vector bundles in
differential geometry. As we shall see,
Theorem~\ref{thm:M.projective.embedded} implies that there exists a
hermitian connection on every regular hermitian module. However, let
us start by recalling the derivation based type of calculus that is
relevant in our context.

Let $\A$ be a unital $\ast$-algebra and let $\d:\A\to\A$ be a derivation. Defining
\begin{align*}
  \d^\ast(f) = \paraa{\d(f^\ast)}^\ast
\end{align*}
it is easy to check that $\d^\ast$ is again a derivation of $\A$.

\begin{definition}
  A set $S\subseteq\Der(\A)$ is $\ast$-closed if $\d^\ast\in S$ for
  all $\d\in S$.
\end{definition}

\noindent
If $V\subseteq\Der(\A)$ is a (complex) vector space which is
$\ast$-closed, then $\ast:V\to V$ represents a real structure on $V$. Hence,
$V$ is the complexification of the real vector space $V_\reals$
generated by the hermitian derivations in $V$, implying that $V$ has a
basis of hermitian derivations (as a complex vector space).

\begin{definition}
  A \emph{Lie pair} is a pair $(\A,\g)$ where $\A$ is a unital
  $\ast$-algebra and $\g$ is a $\ast$-closed Lie algebra of
  derivations on $\A$.
\end{definition}

\noindent
A Lie pair defines a derivation based calculus on the algebra $\A$ and
one introduces connections in a standard way.

\begin{definition}\label{def:connection}
  Let $(\A,\g)$ be a Lie pair and let $M$ be a right $\A$-module.  A
  \emph{connection on $M$} is a map $\nabla:\g\times M\to M$ such that
  \begin{enumerate}
  \item $\nabla_{\d}(m_1+m_2) = \nabla_{\d}m_1+\nabla_{\d}m_2$\label{eq:def.conn.linear}
  \item $\nabla_{\lambda_1\d_1+\lambda_2\d_2}m = \lambda_1\nabla_{\d_1}m+\lambda_2\nabla_{\d_2}m$\label{eq:def.conn.linder}
  \item $\nabla_{\d}(ma) = (\nabla_{\d}m)a + m\d(a)$\label{eq:def.conn.derprop}
  \end{enumerate}
  for $m,m_1,m_2\in M$, $\d,\d_1,\d_2\in\g$,
  $\lambda_1,\lambda_2\in\complex$ and $a\in\A$.
\end{definition}

\noindent
Furthermore, the compatibility with a hermitian form is introduced in
a straight-forward way.

\begin{definition}
  Let $(\A,\g)$ be a Lie pair and let $(M,h)$ be a right hermitian
  $\A$-module. A connection $\nabla$ on $M$ is called \emph{hermitian} if
  \begin{align*}
    \d h(m_1,m_2) = h(\nabla_{\d^\ast}m_1,m_2) + h(m_1,\nabla_{\d}m_2)
  \end{align*}
  for all $m_1,m_2\in M$ and $\d\in\g$. A hermitian connection
  $\nabla$ on a hermitian module $(M,h)$ is said to be
  \emph{compatible with $h$}.
\end{definition}

\noindent
Given a free hermitian module together with an orthogonal projection,
any hermitian connection on the free module induces a hermitian
connection on the corresponding projective module.

\begin{proposition}\label{prop:projected.connection.hermitian}
  Let $(\A,\g)$ be a Lie pair and let $\nabla$ be a hermitian
  connection on the free hermitian module $(\A^n,h)$. If
  $p:\A^n\to\A^n$ is an orthogonal projection, then $p\circ\nabla$ is
  a hermitian connection on $(p(\A^n),h|_{p(\A^n)})$.
\end{proposition}

\begin{proof}
  Since $p$ is a module homomorphism it is clear that properties
  \eqref{eq:def.conn.linear} and \eqref{eq:def.conn.linder} in
  Definition~\ref{def:connection} for $\nablat=p\circ\nabla$ are
  satisfied. To check property \eqref{eq:def.conn.derprop} one
  computes
  \begin{align*}
    \nablat_{\d}(Ua) = p\paraa{\nabla_{\d}(Ua)}
    =p\paraa{(\nabla_\d U)a + U\d(a)}
    =(\nablat_{\d}U)a + U\d(a)
  \end{align*}
  since $p(U)=U$ for $U\in p(\A^n)$. Next, one checks that $\nablat$
  is a hermitian connection; for $U,V\in p(\A^n)$ one obtains
  \begin{align*}
    h\paraa{\nablat_{\d^\ast}U,V}+h\paraa{U,\nablat_{\d}V}
    &= h\paraa{\nabla_{\d^\ast}U,p(V)} + h\paraa{p(U),\nabla_{\d}V}\\
    &=h\paraa{\nabla_{\d^\ast}U,V} + h\paraa{U,\nabla_{\d}V}
      =\d h(U,V)
  \end{align*}
  using that $p$ is orthogonal and that $\nabla$ is a hermitian
  connection on $(\A^n,h)$.
\end{proof}

\noindent
It is straight-forward to parametrize all hermitian connections on a
free regular hermitian module.

\begin{proposition}\label{prop:hermitian.connections.free.module}
  Let $(\A,\g)$ be a Lie pair, let $\{\d_a\}_{a=1}^m$ be a hermitian
  basis of $\g$ and let $\{\eh_i\}_{i=1}^n$ be a basis of the free
  module $\A^n$. A connection $\nabla$ on the regular right
  hermitian $\A$-module $(\A^n,h)$ is hermitian if and only if there exist
  $\gamma_{a,ij}\in\A$ (for $a=1,\ldots,m$ and $i,j=1,\ldots,n$) such
  that $\gamma_{a,ij}^\ast = \gamma_{a,ji}$ and
  \begin{align}
    \nabla_{\d_a}\eh_i = \eh_jh^{jk}\paraa{\tfrac{1}{2}h_{ki}+i\gamma_{a,ki}}
  \end{align}
  where $h_{ij}=h(\eh_i,\eh_j)$ and $h^{ij}h_{jk}=\delta^i_k\mid$.
\end{proposition}

\begin{proof}
  Let $\{\eh_i\}_{i=1}^n$ be a basis of $\A^n$ and let
  $\{\d_a\}_{a=1}^m$ be a hermitian basis of $\g$. A connection $\nabla$ is
  determined by the Christoffel symbols $\Gamma^i_{aj}$ defined as
  \begin{align*}
    \nabla_{\d_a}\eh_i = \eh_j\Gamma^j_{ai.}
  \end{align*}
  Conversely, an arbitrary choice of $\Gamma^i_{aj}\in \A$ defines a
  connection on $\A^n$ via
  \begin{align*}
    \nabla_{\d_a}\paraa{\eh_iU^i} = \eh_i\Gamma^i_{aj}U^j+\eh_i\d_aU^i.
  \end{align*}
  A hermitian connection on $(\A^n,h)$ satisfies
  \begin{align*}
    \d_ah(\eh_i,\eh_j) = h(\nabla_{\d_a}\eh_i,\eh_j)+h(\eh_i,\nabla_{\d_a}\eh_j),
  \end{align*}
  since $\d_a^\ast=\d_a$, giving 
  \begin{align*}
    \d_ah_{ij} = \paraa{\Gamma^k_{ai}}^\ast h_{kj}+h_{ik}\Gamma^k_{aj}\equivalent
    \d_ah_{ij} = \paraa{h_{jk}\Gamma^k_{ai}}^\ast+h_{ik}\Gamma^k_{aj}.
  \end{align*}
  Introducing $\Gammat_{a,ij}=h_{ik}\Gamma^k_{aj}$ the above equation becomes
  \begin{align*}
    \d_a h_{ij} = \Gammat_{a,ij}+\Gammat_{a,ji}^\ast,
  \end{align*}
  and writing $\Gammat_{a,ij} = \tfrac{1}{2}\d_ah_{ij}+i\gamma_{a,ij}$
  the above equation is equivalent to
  \begin{align*}
    \gamma_{a,ij}^\ast=\gamma_{a,ji}.
  \end{align*}
  Hence, $\nabla$ is a metric connection on $\A^n$ if and only if
  there exists $\gamma_{a,ij}\in\A$, such that
  $\gamma_{a,ij}^\ast=\gamma_{a,ji}$, and
  \begin{equation*}
    \Gamma^i_{aj} = \frac{1}{2}h^{ik}\d_ah_{kj}+ih^{ik}\gamma_{a,kj}.\qedhere
  \end{equation*}
\end{proof}

\noindent
Using the above results together with
Theorem~\ref{thm:M.projective.embedded} one can conclude that every
regular hermitian module has a hermitian connection (in fact, there
exist many such connections in general).

\begin{corollary}
  Let $(\A,\g)$ be a Lie pair and let $(M,h)$ be a regular hermitian
  $\A$-module. Then there exists a hermitian connection on $(M,h)$.
\end{corollary}

\begin{proof}
  It follows from Theorem~\ref{thm:M.projective.embedded} that there
  exists a free regular hermitian module $(B,b)$ and a projection
  $\ph:B\to B$ such that $(M,h)\simeq (\ph(B),b|_{\ph(B)})$, and
  Proposition~\ref{prop:hermitian.connections.free.module} implies
  that there exists a hermitian connection $\nabla$ on
  $(B,b)$. Furthermore, it follows from
  Proposition~\ref{prop:projected.connection.hermitian} that
  $\ph\circ\nabla$ is a hermitian connection on
  $(\ph(B),b|_{\ph(B)})$. Since $(M,h)$ is isometric to $(\ph(B),b|_{\ph(B)})$
  we conclude that there exists a hermitian connection on $(M,h)$.
\end{proof}

\noindent
Let us now derive an explicit expression for a hermitian connection on
a regular hermitian module $(p(\A^n),\htilde|_{p(\A^n)})$. To this
end, let $(\A^n,\htilde)$ be a free regular (right) hermitian module
with basis $\{\eh_i\}_{i=1}^n$, and let $p$ be an orthogonal
projection with respect to $\htilde$, written as
\begin{align*}
  p(U) = \eh_i{P^{i}}_jU^j
\end{align*}
for $U=\eh_iU^i\in\A^n$. Moreover, we set $\et_i=\eh_j{P^j}_i$,
implying that the projective module $M=p(\A^n)$ is generated by
$\{\et_i\}_{i=1}^n$. Given a Lie pair $(\A,\g)$ and a hermitian basis
$\{\d_a\}_{a=1}^m$ of $\g$, let $\nablat$ be a hermitian connection on
$(\A^n,\htilde)$, as given in
Proposition~\ref{prop:hermitian.connections.free.module}
\begin{align}
  \nablat_{\d_a}(\eh_iU^i) = \eh_i\d U^i
  + \eh_i\htilde^{ik}\paraa{\tfrac{1}{2}\d_a\htilde_{kj}+i\gamma_{a,kj}}U^j.
\end{align}
Denoting $h=\htilde|_M$, it follows that $\nabla=p\circ\nablat$ is a
connection on $M$ compatible with the hermitian form $h$. One finds that
\begin{align}\label{eq:connection.projective.module}
  \nabla_{\d_a}\et_i = \et_j\d_a{P^{j}}_i+
  \et_j\htilde^{jk}\paraa{\tfrac{1}{2}\d_a\htilde_{kl}+i\gamma_{a,kl}}{P^l}_i.
\end{align}

\subsection{The fuzzy sphere}\label{sec:fuzzy.sphere}

As an example of the concepts developed in the previous section, let us
consider hermitian connections on the fuzzy sphere. For $\hbar>0$,
let $\Stwoh$ denote the unital $\ast$-algebra generated by
\begin{align*}
  X_1=X=X^\ast\qquad X_2=Y=Y^\ast\qquad X_3=Z=Z^\ast
\end{align*}
satisfying
\begin{align}
  [X_i,X_j]=\sum_{k=1}^3i\hbar\eps_{ijk}X_k\qand
  \sum_{k=1}^3X_k^2=X^2+Y^2+Z^2=\mid\label{eq:fuzzy.sphere.def.rel}
\end{align}
where $\eps_{ijk}$ denotes the completely antisymmetric symbol with
$\eps_{123}=1$. In the following we shall assume that any repeated
index is summed over from 1 to 3 unless stated otherwise. Moreover, we
let $\g_{S^2}$ denote the Lie algebra generated by the hermitian inner
derivations
\begin{align*}
  \d_1(f) = \frac{1}{i\hbar}[X,f]\qquad
  \d_2(f) = \frac{1}{i\hbar}[Y,f]\qquad
  \d_3(f) = \frac{1}{i\hbar}[Z,f]
\end{align*}
for $f\in\Stwoh$, satisfying
\begin{align*}
  [\d_i,\d_j]=\sum_{k=1}^3\eps_{ijk}\d_k.
\end{align*}
On the free (right) module $(\Stwoh)^3$, with basis
$\{\eh_i\}_{i=1}^3$, we consider the invertible hermitian form
$\htilde$ defined by
\begin{align}\label{eq:fuzzy.def.htilde}
  \htilde(U,V) = (U_k)^\ast V_k
\end{align}
for $U=\eh_iU_i$ and $V=\eh_iV_i$.

As in classical geometry, the fuzzy sphere has a natural projector in $(\Stwoh)^3$
\begin{align*}
  \Pi(U)=\eh_i\Pi_{ij}U_j = \eh_i\paraa{X_iX_j}U_j
\end{align*}
and one easily checks that
\begin{align*}
  \Pi^2(U) = \eh_i\Pi_{ij}\Pi_{jk}U_k
  =\eh_iX_i\paraa{X_jX_j}X_kU_k
  =\eh_iX_iX_kU_k=\Pi(U),
\end{align*}
due to \eqref{eq:fuzzy.sphere.def.rel}, as well as
\begin{align*}
  \htilde(\Pi(U),V)
  &=(\Pi_{ij}U_j)^\ast V_i = (U_j)^\ast X_jX_iV_i
    =(U_j)^\ast\Pi_{ji}V_i =  \htilde(U,\Pi(V)).
\end{align*}
The complementary projection $P=\id-\Pi$ classically defines the
module of sections of the tangent bundle, and we set
$\TStwoh=P((\Stwoh)^3)$, defining a finitely generated projective
module giving
\begin{align*}
  (\Stwoh)^3 = \TStwoh\oplus\Nh
\end{align*}
where $\Nh=\Pi((\Stwoh)^3)$. Since $P$ is an orthogonal projection, it
follows from
Proposition~\ref{prop:projection.regular.hermitian.module} that
$(\TStwoh,\htilde|_{\TStwoh})$ is a regular hermitian module.  In
matrix notation, the projector may be written as
\begin{align*}
  (P_{ij})=
  \begin{pmatrix}
    \mid-X^2 &  -XY     & -XZ\\
    -YX      & \mid-Y^2 & -YZ\\
    -ZX      & -ZY      & \mid-Z^2
  \end{pmatrix},
\end{align*}
with $\tr P = 3\cdot\mid-X^2-Y^2-Z^2=2\cdot\mid$, reflecting the fact that
the classical tangent space of $S^2$ is 2-dimensional,
and the module $\TStwoh$ is clearly generated by $\et_i=\eh_jP_{ji}$, giving
\begin{align*}
  &\et_1 = P(\eh_1) = (\mid-X^2,-YX,-ZX)\\
  &\et_2 = P(\eh_2) = (-XY,\mid-Y^2,-ZY)\\
  &\et_3 = P(\eh_3) = (-XZ,-YZ,\mid-Z^2).
\end{align*}
Similarly, it is easy to see that $\Nh$ is generated by
$\xi=\eh_iX^i=(X,Y,Z)$. Furthermore, $\xi$ is a basis of $\Nh$ since
\begin{align*}
  &\xi f = 0\implies
  \begin{cases}
    Xf = 0\\
    Yf = 0\\
    Zf = 0
  \end{cases}\implies
  \begin{cases}
    X^2f = 0\\
    Y^2f = 0\\
    Z^2f = 0
  \end{cases}\implies\\
  &\paraa{X^2+Y^2+Z^2}f = 0\implies f = 0,
\end{align*}
showing that $\Nh$ is a free module of rank 1.

In classical geometry, the tangent space of the sphere can be
generated by $e_1=(0,z,-y)$, $e_2=(-z,0,x)$ and $e_3=(y,-x,0)$; for
the fuzzy sphere, these vectors acquire a noncommutative correction.

\begin{proposition}\label{prop:e.et.relation}
  The elements $e_1,e_2,e_3\in(\Stwoh)^3$, defined by $e_i = \eps_{ijk}\et_jX_k$, giving
  \begin{equation}\label{eq:e.in.terms.of.et}
    \begin{split}
      &e_1 = \et_2Z-\et_3Y = (0,Z,-Y)-i\hbar\xi X\\
      &e_2 = \et_3X - \et_1Z = (-Z,0,X)-i\hbar\xi Y\\
      &e_3 = \et_1Y-\et_2X = (Y,-X,0)-i\hbar\xi Z,
    \end{split}
  \end{equation}
  generate the module $\TStwoh$. In particular,
  $\et_i = -\eps_{ijk}e_jX_k+i\hbar e_i$; i.e.
  \begin{equation}\label{eq:et.in.terms.of.e}
    \begin{split}
      &\et_1 = e_3Y - e_2Z + i\hbar e_1\\
      &\et_2 = e_1Z - e_3X + i\hbar e_2\\
      &\et_3 = e_2X - e_1Y + i\hbar e_3.
    \end{split}
  \end{equation}
\end{proposition}

\begin{proof}
  First of all, it is clear that $e_i=\eps_{ijk}\et_jX_k\in\TStwoh$
  since $\et_j\in\TStwoh$ for $j=1,2,3$. Next, let us prove one of the
  relations in equation \eqref{eq:e.in.terms.of.et}:
  \begin{align*}
    \et_2Z&-\et_3Y
    = (-XYZ,Z-Y^2Z,-ZYZ) - (-XZY,-YZY,Y-Z^2Y)\\
    &=\paraa{-X[Y,Z],Z-Y[Y,Z],-Y-Z[Y,Z]}=(0,Z,-Y)-i\hbar(X,Y,Z)X\\
    &=(0,Z,-Y)-i\hbar\xi X;
  \end{align*}
  the remaining relations are proven analogously.  Finally, we show that
  \begin{align*}
    -\eps_{ijk}e_jX_k+i\hbar e_i
    &= \eps_{jik}\eps_{jlm}\et_lX_mX_k+i\hbar e_i
      = \paraa{\delta_{il}\delta_{km}-\delta_{im}\delta_{kl}}\et_lX_mX_k+i\hbar e_i\\
    &=\et_iX_kX_k - \et_kX_iX_k+i\hbar e_i = \et_i-\et_k\paraa{[X_i,X_k]+X_kX_i}+i\hbar e_i\\
    &=\et_i - i\hbar\et_k\eps_{ikl}X_l-\et_kX_kX_i + i\hbar e_i = \et_i,
  \end{align*}
  since $\et_kX_k=0$. Hence, since $\{\et_1,\et_2,\et_3\}$ generate
  $\TStwoh$ it is clear that $\{e_1,e_2,e_3\}$ generate $\TStwoh$.
\end{proof}

\noindent
The hermitian form $h=\htilde|_{\TStwoh}$, in terms of the generators
$\et_1,\et_2,\et_2$ is given by
\begin{align*}
  h_{ij}=\htilde(\et_i,\et_j) = P_{ik}P_{kj}=P_{ij}  
\end{align*}
with $h^{ij}=P_{ij}$, and in terms of $e_1,e_2,e_3$ one has
\begin{align*}
  &h(e_1,e_1) = \mid-X^2-\hbar^2X^2 & &h(e_1,e_2) = -YX-\hbar^2XY\\
  &h(e_2,e_2) = \mid-Y^2-\hbar^2Y^2 & &h(e_1,e_3) = -ZX-\hbar^2XZ\\
  &h(e_3,e_3) = \mid-Z^2-\hbar^2Z^2 & &h(e_2,e_3) = -ZY-\hbar^2YZ,
\end{align*}
which can be written in a more compact form as
\begin{align*}
  g_{ij} = h(e_i,e_j) = P_{ji}-\hbar^2\Pi_{ij}.
\end{align*}
If $\et_i=e_jA_{ji}$ and $e_i=\et_jB_{ji}$ then
$\et_i=\et_kB_{kj}A_{ji}$ and
\begin{align*}
  &g_{ij} = (B_{ki})^\ast h_{kl}B_{lj}\\
  &g^{ij} = A_{ik}h^{kl}(A_{jl})^\ast.
\end{align*}
With $A_{ij}=\eps_{ijk}X_k+i\hbar\delta_{ij}$
(cf. Proposition~\ref{prop:e.et.relation}) one obtains
\begin{align*}
  A_{ik}h^{kl}(A_{jl})^\ast =
  A_{ik}P_{kl}(A_{jl})^\ast = P_{ij}-(1-\hbar^2)\Pi_{ij}+\Pi_{ji}+\hbar^2\delta_{ij}.
\end{align*}
However, since
\begin{align*}
  e_iA_{ik}h^{kl}(A_{jl})^\ast = (1+\hbar^2)e_j+e_i\Pi_{ji},
\end{align*}
due to $e_i\Pi_{ij}=0$ one can equally well take
\begin{align}
  g^{ij} = (1+\hbar^2)\delta_{ij}+\Pi_{ji},
\end{align}
satisfying $e_ig^{ij}g_{jk}=e_k$ and $(g^{ij})^\ast=g^{ji}$.

Let us now derive an explicit expression for hermitian connections on
$(\TStwoh,h)$. First, we prove the following lemma.

\begin{lemma}\label{lemma:etk.dP}
  For $i,j=1,2,3$ one has
  \begin{align*}
    \et_k\d_iP_{kj} = -e_iX_j.
  \end{align*}
\end{lemma}

\begin{proof}
  One computes
  \begin{align*}
    \et_k\d_iP_{kj}
    &= \et_k\d_i\paraa{\delta_{kj}\mid-X_kX_j}
      =-\et_k\d_i\paraa{X_kX_j}\\
    &=-\et_k\paraa{(\d_iX_k)X_j+X_k(\d_iX_j)}
      =-\et_k\eps_{ikl}X_lX_j = -e_iX_j,
  \end{align*}
  since $\et_kX_k=0$ and $\d_iX_k=\eps_{ikl}X_l$.
\end{proof}

\noindent
Now, it follows from eq. \eqref{eq:connection.projective.module} and
Lemma~\ref{lemma:etk.dP} that for arbitrary $\gamma_{i,jk}\in\Stwoh$,
such that $\gamma_{i,jk}^\ast=\gamma_{i,kj}$, a hermitian connection
on $\TStwoh$ is given by
\begin{align*}
  \nabla_{\d_i}\et_j &= -e_iX_j+i\et_k\gamma_{i,kj}-i\et_k\gamma_{i,kl}X_lX_j\\
  \nabla_{\d_i}e_j &= \et_iX_j + i\et_m\eps_{jkl}\gamma_{i,mk}X_l
                     -i\hbar\et_m\paraa{\eps_{imk}+i\gamma_{i,mk}}X_kX_j.
\end{align*}
For instance, choosing $\gamma_{i,jk}=0$ one obtains
\begin{align*}
  \nabla^0_{\d_i}\et_j = -e_iX_j\qtext{and}
  \nabla^0_{\d_i}e_j = \et_iX_j-i\hbar e_iX_j = -\eps_{ikl}e_kX_lX_j
\end{align*}
giving the curvature
\begin{align*}
  R^0(\d_i,\d_j)\et_k
  &= -\eps_{ikl}e_jX_l+\eps_{jkl}e_iX_l+i\hbar\paraa{e_iX_j-e_jX_i}X_k\\
  R^0(\d_i,\d_j)e_k
  &= -e_ih(e_j,e_k)+e_jh(e_i,e_k),
\end{align*}
as well as
\begin{align*}
  R^0_{klij} = h\paraa{e_k,R^0(\d_i,\d_j)e_l} = -h(e_k,e_i)h(e_j,e_l)+h(e_k,e_j)h(e_i,e_l).
\end{align*}
Another choice is given by $\gamma_{i,jk}=i\eps_{ijk}$, clearly
satisfying $\gamma_{i,jk}^\ast=\gamma_{i,kj}$, giving
\begin{align*}
  \nabla^{\eps}_{\d_i}\et_j = \eps_{ijk}\et_k\qquad
  \nabla^{\eps}_{\d_i}e_j = \eps_{ijk}e_k
\end{align*}
and
\begin{align*}
  R^{\eps}(\d_i,\d_j)e_k = R^{\eps}(\d_i,\d_j)\et_k = 0,
\end{align*}
for $i,j,k=1,2,3$.

Let us proceed with another example based on the projection
$P:(\Stwoh)^2\to(\Stwoh)^2$, given by
\begin{align*}
  (P_{ab}) = \frac{1}{\sqrt{4+\hbar^2}}
  \begin{pmatrix}
    \alpha(\hbar)\mid + Z & X+iY\\
    X-iY & \alpha(\hbar)\mid- Z
  \end{pmatrix}
\end{align*}
with $\alpha(\hbar)=\frac{1}{2}(\sqrt{4+\hbar^2}-\hbar)$ and
$\tr P=1-\hbar/\sqrt{4+\hbar^2}$ (tending to $1$ as $\hbar\to 0$). This
is known as the ``monopole projection'', and one easily checks that
$P^2=P$ using that
\begin{align*}
  1+\alpha(\hbar)^2=\alpha(\hbar)\sqrt{4+\hbar^2}\qand
  2\alpha(\hbar)+h=\sqrt{4+\hbar^2}.
\end{align*}
Furthermore, since $P_{ab}^\ast=P_{ba}$ for $a,b=1,2$, this projection
is orthogonal with respect to the hermitian form
\begin{align*}
  \htilde\paraa{\eh_aU^a,\eh_bV^b} = \sum_{a=1}^2(U^a)^\ast(V^a)
\end{align*}
where $\{\eh_1,\eh_2\}$ denotes a basis of $(\Stwoh)^2$. Thus, with
$M=P((\Stwoh)^2)$, it follows that $(M,\htilde|_{M})$ is a regular
hermitian module. Metric connections are given by
\eqref{eq:connection.projective.module}
\begin{align*}
  \nabla_{\d_i}\et_a = \et_b\d_iP_{ba} + i\et_b\gamma_{i,bc}P_{ca},
\end{align*}
and for $\gamma_{i,bc}=0$ one obtains
\begin{alignat*}{2}
  &\nabla_{\d_1}\et_1 = -\et_1Y-i\et_2Z &\qquad &\nabla_{\d_1}\et_2 = i\et_1Z+\et_2Y\\
  &\nabla_{\d_2}\et_1 = \et_1X-\et_2Z & &\nabla_{\d_2}\et_2 = -\et_1Z-\et_2X\\
  &\nabla_{\d_3}\et_1 = i\et_2(X-iY)& &\nabla_{\d_3}\et_2 = -i\et_1(X+iY) 
\end{alignat*}
together with the curvature
\begin{align*}
  R(\d_1,\d_2)\et_1
  &= \tfrac{1}{4+\hbar^2}\parab{i\et_1\paraa{-\{Z,Z\}+\hbar Z}-\et_2\paraa{\{Y,Z\}+i\{Z,X\}}}\\
  R(\d_1,\d_2)\et_2
  &= \tfrac{1}{4+\hbar^2}\parab{\et_1\paraa{\{Y,Z\}-i\{Z,X\}}+i\et_2\paraa{\{Z,Z\}+\hbar Z}}\\
  R(\d_2,\d_3)\et_1
  &=\tfrac{1}{4+\hbar^2}\parab{i\et_1\paraa{-\{Z,X\}+\hbar X}-\et_2\paraa{\{X,Y\}+i\{X,X\}}} \\
  R(\d_2,\d_3)\et_2
  &=\tfrac{1}{4+\hbar^2}\parab{\et_1\paraa{\{X,Y\}-i\{X,X\}}+i\et_2\paraa{\{Z,X\}+\hbar X}} \\
  R(\d_3,\d_1)\et_1
  &=\tfrac{1}{4+\hbar^2}\parab{i\et_1\paraa{-\{Y,Z\}+\hbar Y}-\et_2\paraa{\{Y,Y\}+i\{X,Y\}}} \\
  R(\d_3,\d_1)\et_2
  &= \tfrac{1}{4+\hbar^2}\parab{\et_1\paraa{\{Y,Y\}-i\{X,Y\}}+i\et_2\paraa{\{Y,Z\}+\hbar Y}}
\end{align*}
where $\{A,B\}=AB+BA$ denotes the anticommutator.

Let us mention that the Riemannian geometry of a 3D calculus
over the fuzzy sphere has recently been considered
\cite{mt:lc.fuzzy.sphere}. Although there are several similarities
with the computations above, the approach is quite different and based
on a free module of rank 3 representing the differential forms.  It
would, however, be interesting to find a closer connection between the
two examples.

\section{Embedded noncommutative manifolds}
\label{sec:embedded.noncommutative.manifolds}

\noindent
In this section we shall be interested in hermitian modules arising
from a choice of $n$ elements in the algebra $\A$. One may think about
these elements as analogues of embedding coordinates in $\reals^n$,
inducing a metric on the embedded manifold. A special case of this
construction was used in \cite{ach:nms} to generate noncommutative
analogues of minimal surfaces in Euclidean space, which we will come
back to in Section~\ref{sec:embedded.nms}.

\begin{definition}
  A triple $\Sigma = (\A,\g,\{X^1,\ldots,X^n\})$, where $(\A,\g)$ is a
  Lie pair and $X^1,\ldots,X^n\in\A$ are hermitian elements, is called
  an \emph{embedded noncommutative manifold}.
\end{definition}

\noindent Given an embedded noncommutative manifold
$\Sigma=(\A,\g,\{X^1,\ldots,X^n\})$ and a basis $\{\eh_i\}_{i=1}^n$ of the
free module $\A^n$, we define $\varphi:\g\to\A^n$ by 
\begin{align}\label{eq:def.g.varphi}
  \varphi(\d)= \d X = \eh_i\d X^i
\end{align}
for $\d\in\g$. The (right) module $\TSigma$ generated by the image of
$\varphi$ will be referred to as the \emph{module of vector fields of
  $\Sigma$}. Recall that since $\g$ is $\ast$-closed there exists a
basis of $\g$ consisting of hermitian derivations $\{\d_a\}_{a=1}^m$.
The module $\TSigma$ is clearly generated by $\{\varphi(\d_a)\}_{a=1}^m$, and we
shall in the following write $e_a=\varphi(\d_a)=\eh_i\d_aX^i$, as well
as $\d_ae_b=\eh_i\d_a\d_bX^i$.

Furthermore, let $\hz:\A^n\times A^n\to\A$ denote the hermitian form
given by
\begin{align*}
  \hz(U,V) = \sum_{i=1}^m(U^i)^\ast V^i
\end{align*}
for $U=\eh_iU^i$ and $V=\eh_iV^i$, and the restriction of $\hz$ to
$\TSigma$ will be denoted by $h$; with respect to a hermitian basis $\{\d_a\}_{a=1}^m$ of
$\g$, we write
\begin{align*}
  h_{ab} = h(e_a,e_b) = h\paraa{\varphi(\d_a),\varphi(\d_b)}.
\end{align*}
Let us note that he above framework resembles the concept of
\emph{real calculi} developed in
\cite{aw:curvature.three.sphere,aw:cgb.sphere,atn:minimal.embeddings}. However,
in the current setting there are no assumptions on the reality of the
hermitian form.

In general, we shall be interested in embedded noncommutative
manifolds for which $(\TSigma,h)$ is a regular hermitian module,
implying that one can apply the results of
Section~\ref{sec:lie.pairs}.

\begin{definition}
  An embedded noncommutative manifold
  $\Sigma = (\A,\g,\{X^1,\ldots,X^n\})$ is called \emph{regular} if
  $(\TSigma,h)$ is a regular hermitian module.
\end{definition}

\noindent
In view of Proposition~\ref{prop:huab.equiv.projective}, an embedded
noncommutative manifold is regular if and only if there exists
$h^{ab}\in\A$ such that $e_ah^{ab}h_{bc}=e_c$ for
$a=1,\ldots,m$. Moreover, for a regular embedded noncommutative
manifold, it follows from Theorem~\ref{thm:M.projective.embedded} that
there exist hermitian connections on $(\TSigma,h)$.

Due to the map $\varphi:\g\to\TSigma$, as defined in
\eqref{eq:def.g.varphi}, one can introduce a concept of torsion
freeness of connections (cf. also \cite{aw:curvature.three.sphere}).

\begin{definition}
  A connection $\nabla$ on $\TSigma$ is called \emph{torsion free} if
  \begin{align*}
    \nabla_{\d_1}\varphi(\d_2)-\nabla_{\d_2}\varphi(\d_1) = \varphi([\d_1,\d_2]) 
  \end{align*}
  for all $\d_1,\d_2\in\g$. Furthermore, a torsion free hermitian
  connection on $(\TSigma,h)$ is called a \emph{Levi-Civita
    connection}.
\end{definition}

\noindent
In the following we will prove that there exist Levi-Civita
connections on a regular embedded noncommutative manifold. Let us
start by providing an explicit expression for the projection operator
onto $\TSigma$.

\begin{proposition}\label{prop:projection.embedded.manifold}
  Let $\Sigma = (\A,\g,\{X^1,\ldots,X^n\})$ be a regular embedded
  noncommutative manifold and define $p:\A^n\to\A^n$ as
  \begin{align}\label{eq:p.def.huab}
    p(U)=e_ah^{ab}\hz(e_b,U).    
  \end{align}
  where $h^{ab}$ is such that $e_ah^{ab}h_{bc}=e_c$ for
  $c=1,\ldots,n$.  Then $p(\A^n)=\TSigma$ and $p$ is an orthogonal
  projection with respect to $\hz$.
\end{proposition}

\begin{proof}
  Let us first show that $p^2=p$. For $U\in\A^n$ one obtains
  \begin{align*}
    p^2(U)
    &= e_ah^{ab}\hz\paraa{e_b,e_ph^{pq}\hz(e_q,U)}
      = e_ah^{ab}\hz(e_b,e_p)h^{pq}\hz(e_q,U)\\
    &= e_ah^{ab}h_{bp}h^{pq}\hz(e_q,U)
      = e_ph^{pq}\hz(e_q,U) = p(U)
  \end{align*}
  by using that $e_ah^{ab}h_{bc}=e_c$. Moreover, since
  $\Im(p)\subseteq \TSigma$ and
  \begin{align*}
    p(e_a) = e_bh^{bc}\hz(e_c,e_a)
    = e_bh^{bc}h_{ca} = e_a
  \end{align*}
  we conclude that $p(\A^n)\simeq\TSigma$. Finally, let us show that $p$ is
  orthogonal with respect to the hermitian form $\hz$. One computes
  \begin{align*}
    \hz\paraa{p(U),V}
    &=\hz(e_b,U)^\ast(h^{ab})^\ast\hz(e_a,V)
      =\hz(U,e_b)h^{ba}\hz(e_a,V)\\
    &=\hz\paraa{U,e_bh^{ba}\hz(e_a,V)}
    =\hz\paraa{U,p(V)}
  \end{align*}
  for $U,V\in\A^n$.  
\end{proof}

\noindent
Having the projection at hand allows us to construct hermitian
connections on $(\TSigma,h)$ by using Proposition~\ref{prop:hermitian.connections.free.module}.

\begin{proposition}\label{prop:TSigma.metric.connection}
  Let $\Sigma = (\A,\g,\{X^1,\ldots,X^n\})$ be a regular embedded
  noncommutative manifold and let $\{\d_a\}_{a=1}^m$ be a hermitian
  basis of $\g$. For any $\gamma_{a,ij}\in\A$ such that
  $\gamma_{a,ij}^\ast=\gamma_{a,ji}$,
  \begin{align*}
    \nabla_{\d_a}(e_bm^b) = e_b\d_am^b + e_ch^{cp}\paraa{h^0(e_p,\d_ae_b)+i\gamma_{a,pb}}m^b,
  \end{align*}
  where $\gamma_{a,bc} = (e_b^i)^\ast\gamma_{a,ij}e_c^j$, defines a
  hermitian connection on $(\TSigma,h)$.
\end{proposition}

\begin{proof}
  For $h^0_{ij}=\delta_{ij}\mid$ (implying $\d_ah_{ij}=0$), let
  \begin{align*}
    \nabla^0_{\d_a}(\eh_iU^i) = \eh_i\d_aU^i + i\eh_i\delta^{ij}\gamma_{a,jk}U^k
  \end{align*}
  be a metric connection as constructed in
  Proposition~\ref{prop:hermitian.connections.free.module}.  Since
  $\Sigma$ is a regular embedded noncommutative manifold, one can use
  the projection $p:\A^n\to A^n$, defined in \eqref{eq:p.def.huab},
  and set
  \begin{align*}
    \nabla = p\circ\nabla^0.
  \end{align*}
  It follows from
  Proposition~\ref{prop:projected.connection.hermitian} and
  Proposition~\ref{prop:projection.embedded.manifold} that $\nabla$ is
  a hermitian connection on $(\TSigma,h)$.  Let us now compute
  $\nabla_{\d_a}e_b$
  \begin{align*}
    \nabla_{\d_a}e_b
    &= e_ch^{cp}h^0\paraa{e_p,\nabla^0_{\d_a}e_b}
    = e_ch^{cp}h^0\paraa{e_p,\eh_i\d_ae_b^i
      +i\eh_i\delta^{ij}\gamma_{a,jk}e_b^k}\\
    &= e_ch^{cp}\parab{h^0(e_p,\d_ae_b)+i(e_p^j)^\ast\gamma_{a,jk}e_b^k}
      = e_ch^{cp}\parab{h^0(e_p,\d_ae_b)+i\gamma_{a,pb}}
  \end{align*}
  where $\gamma_{a,pb}=(e_p^j)^\ast\gamma_{a,jk}e_b^k$.
\end{proof}

\noindent 
It turns out that requiring extra symmetry properties of
$\gamma_{a,bc}$ is enough to guarantee that the hermitian connection
constructed in Proposition~\ref{prop:TSigma.metric.connection} is a
Levi-Civita connection.

\begin{theorem}\label{thm:Levi.Civita}
  Let $\Sigma = (\A,\g,\{X^1,\ldots,X^n\})$ be a regular embedded
  noncommutative manifold and let $\{\d_a\}_{a=1}^m$ be a hermitian
  basis of $\g$. For any $\gamma_{a,ij}\in\A$ such that
  $\gamma_{a,ij}^\ast=\gamma_{a,ji}$ and
  $\gamma_{a,bc}=\gamma_{c,ba}$, where
  $\gamma_{a,bc}=(e_b^i)^\ast\gamma_{a,ij}e_c^j$,
  \begin{align}\label{eq:levi.civita.connection}
    \nabla_{\d_a}(e_bm^b) = e_b\d_am^b +
    e_c\hsigma^{cp}\paraa{h^0(e_p,\d_ae_b)
    +i\gamma_{a,pb}}m^b
  \end{align}
  defines a Levi-Civita connection on $(\TSigma,h)$. In particular,
  choosing $\gamma_{a,ij}=0$, we conclude that there exists a
  Levi-Civita connection on every regular embedded noncommutative
  manifold.
\end{theorem}

\begin{proof}
  It follows from Proposition~\ref{prop:TSigma.metric.connection} that
  \eqref{eq:levi.civita.connection} defines a metric connection on
  $\TSigma$. The connection is torsion free if
  $\nabla_{\d_a}e_b-\nabla_{\d_b}e_a-e_cf_{ab}^c=0$ where
  $[\d_a,\d_b]=f_{ab}^c\d_c$. Noting that
  \begin{align*}
    h^0(e_p,\d_ae_b)
    &= h^0\paraa{e_p,\eh_i\d_a\d_bX^i}
      =h^0\paraa{e_p,\eh_i\d_b\d_aX^i+\eh_if_{ab}^c\d_cX^i}\\
    &= h^0(e_p,\d_be_a) + h^0(e_p,e_c)f_{ab}^c,
  \end{align*}
  one obtains
  \begin{align*}
    \nabla_{\d_a}e_b-\nabla_{\d_b}e_a-e_cf_{ab}^c
    &= e_ch^{cp}\paraa{h(e_p,e_r)f_{ab}^r+i\gamma_{a,pb}-i\gamma_{b,pa}}-e_cf_{ab}^c\\
    &=ie_ch^{cp}\paraa{\gamma_{a,pb}-\gamma_{b,pa}} = 0
  \end{align*}
  if $\gamma_{a,pb}=\gamma_{b,pa}$.
\end{proof}

\begin{remark}\label{rem:gamma.symmetric}
  Note that $\gamma_{a,bc}$ in Theorem~\ref{thm:Levi.Civita} fulfills
  $\gamma_{a,bc}^\ast=\gamma_{a,cb}$ and
  $\gamma_{a,bc}=\gamma_{c,ba}$, implying that
  \begin{align*}
    \gamma_{a,cb} = \gamma_{b,ca} = \gamma_{b,ac}^\ast
    =\gamma_{c,ab}^\ast = \gamma_{c,ba} = \gamma_{a,bc}
  \end{align*}
  which, together with $\gamma_{a,bc}=\gamma_{c,ba}$, implies that
  $\gamma_{a,bc}$ is completely symmetric in all three indices; i.e.
  $\gamma_{a,bc} = \gamma_{b,ac} = \gamma_{c,ba} = \gamma_{a,cb}$.
  Consequently, if $\{e_a\}_{a=1}^m$ is a basis of $\TSigma$, one can
  define a Levi-Civita connection via
  \eqref{eq:levi.civita.connection} for any choice of hermitian
  $\gamma_{a,bc}\in\A$ such that $\gamma_{a,bc}$ is symmetric in all
  three indices. Hence, in this case the Levi-Civita connection is
  determined by the choice of
  \begin{align*}
    \frac{(m+2)!}{3!(m-1)!}
  \end{align*}
  hermitian elements from $\A$.
\end{remark}

\section{Embedded noncommutative minimal surfaces}
\label{sec:embedded.nms}

\noindent
Let us now turn to a particular class of noncommutative embedded
manifolds appearing in \cite{ach:nms}. For $\hbar>0$, let $\Wh$ denote
the first Weyl algebra, i.e $\Wh$ is generated by $U$ and $V$,
satisfying $[U,V]=i\hbar\mid$; moreover, $\Wh$ is a $\ast$-algebra with the
involution defined as $U^\ast=U$ and $V^\ast=V$. The Weyl
algebra satisfies the Ore condition implying that it has a fraction
field, which will be denoted by $\Fh$.

Let $\g_2$ denote the (abelian) Lie algebra generated by the two
hermitian derivations
\begin{align*}
  \delta_1(f)\equiv\d_u(f) = \frac{1}{i\hbar}[f,V]\qquad
  \delta_2(f)\equiv\d_v(f) = \frac{1}{i\hbar}[f,U]
\end{align*}
satisfying $[\d_u,\d_v]=0$. We shall also use a complementary
description given by the complex combination $\Lambda = U+iV$, giving
$[\Lambda,\Lambda^\ast]=2\hbar\mid$, and set
\begin{align*}
  \d_1 = \d = \tfrac{1}{2}\paraa{\d_u - i\d_v}=\tfrac{1}{2\hbar}[\cdot,\Lambda^\ast]\qquad
  \d_2 = \db = \tfrac{1}{2}\paraa{\d_u + i\d_v}=\tfrac{1}{2\hbar}[\cdot,\Lambda].
\end{align*}
Let us now briefly recall how one can formulate the equations for a
minimal surface in $\reals^n$ in terms of a Poisson structure on
$\Sigma$.  Let $u,v$ denote local coordinates on a surface $\Sigma$
and let $\xv:\Sigma\to\reals^n$ be an embedding of $\Sigma$ into
$\reals^n$. The embedding is minimal if the embedding coordinates are
harmonic; i.e. $\Delta(x^i)=0$ for $i=1,\ldots,n$, where $\Delta$
denotes the Laplace-Beltrami operator with respect to the induced
metric on $\Sigma$ (cf. \cite{dhko:minimalsurfacesI}).

It is known that most of Riemannian geometry of embedded (almost)
K\"ahler manifolds can be formulated in terms of the Poisson structure
on the manifold (see \cite{ahh:multilinear,ah:pseudoRiemannian}).
For instance, assuming that the coordinates $u,v$ are isothermal,
i.e. there exists a function $\E(u,v)$ such that the metric
$g_{ab}=\E(u,v)\delta_{ab}$, and that $\{u,v\}=1$, the
Laplace-Beltrami operator on an embedded surface becomes
\begin{align*}
  \Delta(f) = \frac{1}{\E}\{\{f,u\},u\} + \frac{1}{\E}\{\{f,v\},v\}
\end{align*}
for $f\in C^\infty(\Sigma)$ (cf. \cite[Proposition 1.1]{ach:nms}).
Hence, one may formulate the conditions for an embedding (into
$\reals^n$) to be minimal in isothermal coordinates as
\begin{align*}
  &\{\{x^i,u\},u\} + \{\{x^i,v\},v\}=0\quad\text{for }i=1,\ldots,n\\
  &(\d_u\xv)\cdot(\d_u\xv) = (\d_v\xv)\cdot(\d_v\xv) \qquad
    (\d_u\xv)\cdot(\d_v\xv)=0
\end{align*}
(cf. \cite[p. 76]{dhko:minimalsurfacesI}). This characterization is
carried over to the noncommutative setting in \cite{ach:nms} as
follows (properly adapted to the current terminology).

\begin{definition}\label{def:nms}
  An embedded noncommutative manifold $\Sigma = (\Fh,\g_2,\{X^1,\ldots,X^n\})$ is
  called a \emph{noncommutative minimal surface} if
  \begin{align*}
    &\d_u^2X^i + \d_v^2X^i = 0\qquad\text{for }i=1,\ldots,n\\
    &h(e_u,e_u)=h(e_v,e_v)\equiv\E\qquad
    -h(e_u,e_v)^\ast = h(e_u,e_v) = i\F
  \end{align*}
\end{definition}

\noindent
Thus, for a noncommutative minimal surface, the metric is given by
\begin{align*}
  h_{ab} = h(e_a,e_b) =
  \begin{pmatrix}
    \E & i\F \\
    -i\F & \E
  \end{pmatrix}
\end{align*}
with $\E^\ast=\E$ and $\F^\ast=\F$. Since $\Fh$ is a field, as long as
$\E+\F\neq 0$ and $\E-\F\neq 0$, it is clear that the metric is
invertible, with inverse
\begin{align*}
  h^{-1}
  = (\E+\F)^{-1}
  \begin{pmatrix}
    \E & -i\F \\
    i\F & \E
  \end{pmatrix}
          (\E-\F)^{-1},
\end{align*}
implying that $\TSigma$ is a free module and that
$(\Fh,\g_2,\{X^1,\ldots,X^n\})$ is a regular embedded noncommutative
manifold.  Writing
\begin{align*}
  \Phi_1=\Phi=\d X=\tfrac{1}{2}(e_u-ie_v)
  \qquad \Phi_2=\Phib=\db X = \tfrac{1}{2}(e_u+ie_v).
\end{align*}
the conditions in Definition~\ref{def:nms} can be expressed as
\begin{align}\label{eq:nms.def.complex}
  \db\Phi^i = 0\qtext{and}
  h(\Phib,\Phi) = 0,
\end{align}
and the metric becomes
\begin{align*}
  H_{ab} = h(\Phi_a,\Phi_b) =
  \begin{pmatrix}
    S & 0 \\
    0 & T
  \end{pmatrix}
        =\frac{1}{2}
        \begin{pmatrix}
          \E + \F & 0 \\
          0 & \E-\F 
        \end{pmatrix}.
\end{align*}
The formulation in \eqref{eq:nms.def.complex} gives us a rather easy
way to construct examples of noncommutative minimal surfaces. Namely,
if $\Phi^i$ is a rational function in $\Lambda$ then $\db\Phi^i=0$,
and the condition $h(\Phib,\Phi)=0$ is equivalent to
\begin{align*}
  (\Phi^1)^2+(\Phi^2)^2 + \cdots + (\Phi^n)^2 = 0.
\end{align*}
For instance, a class of algebraic minimal surfaces is generated by
\begin{align*}
  \Phi=\eh_1\Phi^1+\eh_2\Phi^2+\eh_3\Phi^3 
\end{align*}
with
\begin{align*}
  \Phi^1 = (\mid-\Lambda^2)F(\Lambda)\quad
  \Phi^2 = i(\mid+\Lambda^2)F(\Lambda)\quad
  \Phi^3 = 2\Lambda F(\Lambda)
\end{align*}
for arbitrary $F(\Lambda)\in\Wh$. Note that since $\Phi^i$ is a
polynomial in $\Lambda$, one can always find hermitian $X^i\in\Wh$
such that $\d X^i=\Phi^i$. Moreover, since $\Phi^i$ and $\d\Phi^i$ are
both polynomials in $\Lambda$, they commute;
i.e. $[\Phi^i,\d\Phi^i]=0$ for $i=1,2,3$. The simplest example is
given by $F(\Lambda)=\mid$, corresponding to the classical Enneper
surface, yielding a noncommutative minimal surface
$(\Fh,\g_2,\{X^1,X^2,X^3\})$ with
\begin{align*}
  X^1 &= \Lambda-\tfrac{1}{3}\Lambda^3+\Lambda^\ast-\tfrac{1}{3}(\Lambda^\ast)^3\\
  X^2 &= i\paraa{\Lambda+\tfrac{1}{3}\Lambda^3-\Lambda^\ast-(\Lambda^\ast)^3}\\
  X^3 &= \Lambda^2+(\Lambda^\ast)^2,
\end{align*}
giving
\begin{align*}
  S &= 2\paraa{\mid+(\Lambda^\ast)^2\Lambda^2+2\Lambda^\ast\Lambda}\\
  T &= 2\paraa{\mid+\Lambda^2(\Lambda^\ast)^2+2\Lambda\Lambda^\ast}.
\end{align*}
Since a noncommutative minimal surface is a regular embedded
noncommutative manifold, the existence of a Levi-Civita connection is
guaranteed. As an illustration of the concepts in
Section~\ref{sec:embedded.noncommutative.manifolds}, let us derive an
explicit expression for a Levi-Civita connection on a noncommutative
minimal surface $\Sigma=(\Fh,\g_2,\{X^1,\ldots,X^n\})$, satisfying
$[\Phi^i,\d\Phi^i]$ for $i=1,\ldots,n$.

In terms of $\Phi,\Phib$ one may write the projection onto $\TSigma$
as
\begin{align*}
  p(U) = e_ah^{ab}h^0(e_b,U) = \Phi_aH^{ab}h^0(\Phi_b,U)
\end{align*}
giving the Levi-Civita connection in Theorem~\ref{thm:Levi.Civita}
(cf. eq \eqref{eq:levi.civita.connection}) as
\begin{align}\label{eq:nms.lcconnection.complex}
  \nabla_{\d_a}\Phi_b = \Phi_rH^{rs}h^0(\Phi_s,\d_a\Phi_b)
  +i\Phi_rH^{rs}\gammat_{a,sb}
\end{align}
for arbitrary $\gammat_{a,bc}\in\Fh$ satisfying
\begin{align*}
  \gammat_{1,bc}^\ast = \gammat_{2,cb}\qand
  \gammat_{a,bc}=\gammat_{c,ba}
\end{align*}
for $a,b,c=1,2$. The above symmetries imply that there are two
independent components of $\gammat_{a,bc}$ (e.g. $\gammat_{1,11}$ and
$\gammat_{1,21}$).

\begin{lemma}\label{lemma:h0.Phib.dPHi}
  Let $\Sigma=(\Fh,\g_2,\{X^1,\ldots,X^n\})$ be a noncommutative
  minimal surface. If $[\Phi^i,\d\Phi^i]=0$ for $i=1,\ldots,n$ then
  \begin{align*}
    h(\Phib,\d\Phi) = h(\Phi,\db\Phib)=0. 
  \end{align*}
\end{lemma}

\begin{proof}
  Using that $\Phi^i(\d\Phi^i)=(\d\Phi^i)\Phi^i$ one finds that
  \begin{align}\label{eq:h.Phib.dPhi}
    h(\Phib,\d\Phi)
    = \sum_{i=1}^n\Phi^i\d\Phi^i
    = \sum_{i=1}^n(\d\Phi^i)\Phi^i
    = \sum_{i=1}^n(\db\Phib^i)^\ast\Phi^i
    =h(\db\Phib,\Phi).
  \end{align}
  Since $h(\Phib,\Phi)=0$ for a noncommutative minimal surface, it
  follows that
  \begin{align*}
    0 = \d h(\Phib,\Phi)
    = h(\db\Phib,\Phi)+h(\Phib,\d\Phi)=2h(\Phib,\d\Phi)
    =2h(\db\Phib,\Phi),
  \end{align*}
  using \eqref{eq:h.Phib.dPhi}, giving
  $h(\Phib,\d\Phi) = h(\Phi,\db\Phib)=0$.
\end{proof}

\begin{proposition}
 Let $\Sigma=(\Fh,\g_2,\{X^1,\ldots,X^n\})$ be a noncommutative
 minimal surface such that $[\Phi^i,\d\Phi^i]=0$ for $i=1,\ldots,n$. Then
 \begin{equation}\label{eq:nms.lc.connection}
   \begin{split}
     \nabla_{\d}\Phi &= \Phi S^{-1}\d S + i\Phi S^{-1}\gammat_1+i\Phib T^{-1}\gammat_2\\
     \nabla_{\db}\Phib &= \Phib T^{-1}\db T + i\Phi S^{-1}\gammat_2^\ast+i\Phib T^{-1}\gammat_1^\ast\\
     \nabla_{\d}\Phib
     &= i\Phi S^{-1}\gammat_1^\ast + i\Phib T^{-1}\gammat_1
     = \nabla_{\db}\Phi
   \end{split}   
 \end{equation}
 defines a Levi-Civita connection on $(\TSigma,h)$ for arbitrary $\gammat_1,\gammat_2\in\Fh$.
\end{proposition}

\begin{proof}
  Let us denote the two independent components of $\gammat_{a,bc}$ by
  $\gammat_1=\gammat_{1,11}$ and $\gammat_2=\gammat_{1,21}$.  Since
  $\db\Phi^i=\d(\Phi^i)^\ast=0$ one finds that
  \begin{align*}
    \d S
    &= \d h^0(\Phi,\Phi)
    = \sum_{i=1}^n\d\paraa{(\Phi^i)^\ast\Phi^i}
    = \sum_{i=1}^n(\Phi^i)^\ast\d\Phi^i = h^0(\Phi,\d\Phi)\\
    \db T
    &= \db h^0(\Phib,\Phib)
      = \sum_{i=1}^n\db\paraa{\Phi^i(\Phi^i)^\ast}
      = \sum_{i=1}^n\Phi^i\db(\Phi^i)^\ast
      = h^0(\Phib,\db\Phib).
  \end{align*}
  Using $h(\Phib,\d\Phi) = h(\Phi,\db\Phib)=0$ (cf.
  Lemma~\eqref{eq:h.Phib.dPhi}) together with
  \eqref{eq:nms.lcconnection.complex}, one obtains
  \begin{align*}
    \nabla_{\d}\Phi
    &=\Phi_rH^{rs}h^0(\Phi_s,\d\Phi)+i\Phi_rH^{rs}\gammat_{1,s1}\\
    &= \Phi S^{-1}h^0(\Phi,\d\Phi)
    +i\Phi S^{-1}\gammat_{1,11}+i\Phib T^{-1}\gammat_{1,21}\\
    &=\Phi S^{-1}\d S +i\Phi S^{-1}\gammat_1+i\Phib T^{-1}\gammat_{2}
  \end{align*}
  and
  \begin{align*}
    \nabla_{\db}\Phib
    &=\Phi_rH^{rs}h^0(\Phi_s,\db\Phib)+i\Phi_rH^{rs}\gammat_{2,s2}\\
    &=\Phib T^{-1}h^0(\Phib,\db\Phib)
      +i\Phi S^{-1}\gammat_{2,12}+i\Phib T^{-1}\gammat_{2,22}\\
    &=\Phib T^{-1}\db T + i\Phi S^{-1}\gammat_{2}^\ast + i\Phib T^{-1}\gammat_{1}^\ast
  \end{align*}
  as well as
  \begin{align*}
    \nabla_{\db}\Phi = \nabla_{\d}\Phib
    &= i\Phi S^{-1}\gammat_{1,12} + i\Phib T^{-1}\gammat_{1,22}
    = i\Phi S^{-1}\gammat_1^\ast + i\Phib T^{-1}\gammat_1,
  \end{align*}
  proving the formulas in \eqref{eq:nms.lc.connection}.
\end{proof}

\noindent
In the particularly simple case when $\gammat_1=\gammat_2=0$ one
obtains
\begin{align*}
  \nabla_{\d}\Phi = \Phi S^{-1}\d S\qquad
  \nabla_{\db}\Phib = \Phib T^{-1}\db T\qquad
  \nabla_{\d}\Phib = \nabla_{\db}\Phi = 0
\end{align*}
as well as
\begin{align*}
  R(\d,\db)\Phi
  &= \nabla_{\d}\nabla_{\db}\Phi-\nabla_{\db}\nabla_{\d}\Phi
    =-\Phi\db\paraa{S^{-1}\d S}\\
  R(\d,\db)\Phib
  &= \nabla_{\d}\nabla_{\db}\Phib-\nabla_{\db}\nabla_{\d}\Phib
    =\Phib\d\paraa{T^{-1}\db T}.
\end{align*}

\section*{Acknowledgements}

\noindent
J.A. is supported by the Swedish Research Council grant no. 2017-03710.

\bibliographystyle{alpha}
\bibliography{references}  

\def\polhk#1{\setbox0=\hbox{#1}{\ooalign{\hidewidth
  \lower1.5ex\hbox{`}\hidewidth\crcr\unhbox0}}}
\begin{thebibliography}{DHKW92}

\bibitem[AC10]{ac:ncgravitysolutions}
P.~Aschieri and L.~Castellani.
\newblock Noncommutative gravity solutions.
\newblock {\em J. Geom. Phys.}, 60(3):375--393, 2010.

\bibitem[ACH16]{ach:nms}
J.~Arnlind, J.~Choe, and J.~Hoppe.
\newblock Noncommutative minimal surfaces.
\newblock {\em Lett. Math. Phys.}, 106(8):1109--1129, 2016.

\bibitem[AH10]{ah:dmsa}
Joakim Arnlind and Jens Hoppe.
\newblock Discrete minimal surface algebras.
\newblock {\em SIGMA}, 6:042, 2010.

\bibitem[AH12]{ah:quantizedMinimal}
J.~Arnlind and J.~Hoppe.
\newblock {The world as quantized minimal surfaces}.
\newblock \texttt{arXiv:1211.1202}, 2012.

\bibitem[AH14]{ah:pseudoRiemannian}
J.~Arnlind and G.~Huisken.
\newblock Pseudo-{R}iemannian geometry in terms of multi-linear brackets.
\newblock {\em Lett. Math. Phys.}, 104(12):1507--1521, 2014.

\bibitem[AH18]{ah:catenoid}
J.~Arnlind and C.~Holm.
\newblock A noncommutative catenoid.
\newblock {\em Lett. Math. Phys.}, 108(7):1601--1622, 2018.

\bibitem[AHH12]{ahh:multilinear}
J.~Arnlind, J.~Hoppe, and G.~Huisken.
\newblock Multi-linear formulation of differential geometry and matrix
  regularizations.
\newblock {\em J. Differential Geom.}, 91(1):1--39, 2012.

\bibitem[AHK19]{ahk:qms}
J.~Arnlind, J.~Hoppe, and M.~Kontsevich.
\newblock Quantum minimal surfaces.
\newblock \texttt{arXiv:1903.10792}, 2019.

\bibitem[AHT04]{aht:spinning}
Joakim Arnlind, Jens Hoppe, and Stefan Theisen.
\newblock Spinning membranes.
\newblock {\em Phys. Lett. B}, 599(1-2):118--128, 2004.

\bibitem[Asc20]{a:lc.braided}
P.~Aschieri.
\newblock Cartan structure equations and {L}evi-{C}ivita connection in braided
  geometry.
\newblock \texttt{arXiv:2006.02761}, 2020.

\bibitem[ATN21]{atn:minimal.embeddings}
J.~Arnlind and A.~Tiger~Norkvist.
\newblock Noncommutative minimal embeddings and morphisms of
  pseudo-{R}iemannian calculi.
\newblock {\em J. Geom. Phys.}, 159:103898, 17, 2021.

\bibitem[AW17a]{aw:cgb.sphere}
J.~Arnlind and M.~Wilson.
\newblock On the {C}hern-{G}auss-{B}onnet theorem for the noncommutative
  4-sphere.
\newblock {\em J. Geom. Phys.}, 111:126--141, 2017.

\bibitem[AW17b]{aw:curvature.three.sphere}
J.~Arnlind and M.~Wilson.
\newblock {R}iemannian curvature of the noncommutative 3-sphere.
\newblock {\em J. Noncommut. Geom.}, 11(2):507--536, 2017.

\bibitem[BGL20]{bgl:lc.nc.differential.calculi}
J.~Bhowmick, D.~Goswami, and G.~Landi.
\newblock Levi-{C}ivita connections and vector fields for noncommutative
  differential calculi.
\newblock {\em Internat. J. Math.}, 31(8):2050065, 23, 2020.

\bibitem[BM11]{bm:starCompatibleConnections}
E.~J. Beggs and S.~Majid.
\newblock {$*$}-compatible connections in noncommutative {R}iemannian geometry.
\newblock {\em J. Geom. Phys.}, 61(1):95--124, 2011.

\bibitem[CT99]{ct:holomorphic.curves.matrices}
L.~Cornalba and W.~Taylor, IV.
\newblock Holomorphic curves from matrices.
\newblock {\em Nuclear Phys. B}, 536(3):513--552, 1999.

\bibitem[CT11]{ct:gaussBonnet}
A.~Connes and P.~Tretkoff.
\newblock The {G}auss-{B}onnet theorem for the noncommutative two torus.
\newblock In {\em Noncommutative geometry, arithmetic, and related topics},
  pages 141--158. Johns Hopkins Univ. Press, Baltimore, MD, 2011.

\bibitem[DHKW92]{dhko:minimalsurfacesI}
U.~Dierkes, S.~Hildebrandt, A.~K{\"u}ster, and O.~Wohlrab.
\newblock {\em Minimal surfaces. {I}}, volume 295 of {\em Grundlehren der
  Mathematischen Wissenschaften}.
\newblock Springer-Verlag, Berlin, 1992.
\newblock Boundary value problems.

\bibitem[DLL15]{dll:sigma.model.solitons}
L.~Dabrowski, G.~Landi, and F.~Luef.
\newblock Sigma-model solitons on noncommutative spaces.
\newblock {\em Lett. Math. Phys.}, 105(12):1663--1688, 2015.

\bibitem[FK13]{fk:scalarCurvature}
F.~Fathizadeh and M.~Khalkhali.
\newblock Scalar curvature for the noncommutative two torus.
\newblock {\em J. Noncommut. Geom.}, 7(4):1145--1183, 2013.

\bibitem[Hop82]{h:phdthesis}
Jens Hoppe.
\newblock {\em Quantum Theory of a Massless Relativistic Surface and a
  Two-dimensional Bound State Problem}.
\newblock PhD thesis, Massachusetts Institute of Technology, 1982.

\bibitem[Lan18]{l:twisted.sigma.model}
G.~Landi.
\newblock Twisted sigma-model solitons on the quantum projective line.
\newblock {\em Lett. Math. Phys.}, 108(8):1955--1983, 2018.

\bibitem[LTM20]{mt:lc.fuzzy.sphere}
E.~Lira~Torres and S.~Majid.
\newblock Quantum gravity and {R}iemannian geometry on the fuzzy sphere.
\newblock \texttt{arXiv:2004.14363}, 2020.

\bibitem[MR11]{mr:nc.sigma.model}
V.~Mathai and J.~Rosenberg.
\newblock A noncommutative sigma-model.
\newblock {\em J. Noncommut. Geom.}, 5(2):265--294, 2011.

\bibitem[MW18]{mw:quantum.koszul}
S.~Majid and L.~Williams.
\newblock Quantum koszul formula on quantum spacetime.
\newblock {\em Journal of Geometry and Physics}, 129:41 -- 69, 2018.

\bibitem[Ros13]{r:leviCivita}
J.~Rosenberg.
\newblock {L}evi-{C}ivita's theorem for noncommutative tori.
\newblock {\em SIGMA}, 9:071, 2013.

\bibitem[SS17]{ss:covariant.spheres}
M.~Sperling and H.~C. Steinacker.
\newblock Covariant 4-dimensional fuzzy spheres, matrix models and higher spin.
\newblock {\em J. Phys. A}, 50(37):375202, 45, 2017.

\bibitem[Syk16]{s:fuzzy.construction.kit}
A.~Sykora.
\newblock The fuzzy space construction kit.
\newblock \texttt{arXiv:1610.01504}, 2016.

\end{thebibliography}

\end{document}